\documentclass[11pt]{article}

\usepackage{amsmath,amsbsy,amssymb,latexsym,amsthm}

\newtheorem{Theorem}{Theorem}
\newtheorem{Definition}{Definition}
\newtheorem{Remark}{Remark}
\newtheorem{Lemma}{Lemma}

\renewenvironment{abstract}{\begin{center}
\begin{minipage}[c]{12cm} {\begin{center}\bf Abstract \end{center}}} {\end{minipage}
\end{center}}
\newenvironment{keywords}{\begin{center}
\begin{minipage}[c]{12cm} {\bf Key words:}} {\end{minipage}
\end{center}}
\newenvironment{msc}{\begin{center}
\begin{minipage}[c]{12cm} {\bf Mathematics Subject Classification:}} {\end{minipage}
\end{center}}

\begin{document}

\title{The isoperimetric problem for H\"{o}lderian curves}

\author{Ricardo Almeida\\
\texttt{ricardo.almeida@ua.pt}
\and
Delfim F. M. Torres\\
\texttt{delfim@ua.pt}}

\date{Department of Mathematics\\
University of Aveiro\\
3810-193 Aveiro, Portugal}

\maketitle


\begin{abstract}
We prove a necessary stationary condition
for non-differentiable isoperimetric variational problems
with scale derivatives, defined on the class of H\"{o}lder continuous functions.
\end{abstract}

\begin{msc}
49K05, 26B05, 39A12.
\end{msc}

\begin{keywords}
scale calculus, isoperimetric problem, non differentiability.
\end{keywords}


\section{Introduction}

An analogue of differentiable calculus for H\"{o}lder continuous functions has been recently developed by J.~Cresson,
by substituting the classical notion of derivative by
a new complex operator, called the \emph{scale derivative} \cite{Cresson2}. A Leibniz rule similar to the classical one
is proved, and with it a generalized Euler-Lagrange equation,
valid for nonsmooth curves, is obtained \cite{Cresson2}.
The new calculus of variations find applications in scale-relativity theory, and some applications are given
to Hamilton's principle of least action and to nonlinear Schr\"{o}dinger equations \cite{Cresson1,Cresson2,Cresson3}.

In this note we introduce the isoperimetric problem
for H\"{o}lder continuous curves in Cresson's setting.
Section~\ref{sec:prm} reviews the quantum calculus of J.~Cresson, fixing some typos found in \cite{Cresson2}.
Main results are given in Section~\ref{sec:mr},
where the non differentiable isoperimetric problem
is formulated and respective stationary condition proved (see Theorem~\ref{MainTeo}). We end with Section~\ref{sec:ex},
illustrating the applicability of our Theorem~\ref{MainTeo}
to a simple example that has an H\"{o}lder continuous extremal,
which is not differentiable in the classical sense.


\section{Preliminaries}
\label{sec:prm}

In this section we review the quantum calculus \cite{Cresson1,Cresson2},
which extends the classical differential calculus
to nonsmooth continuous curves.
As usual, we denote by $C^0$ the set of continuous real valued functions defined on $\mathbb R$.

\begin{Definition} (\cite{Cresson2}) Let $f\in C^0$ and $\epsilon>0$. The $\epsilon-$left and $\epsilon-$right quantum derivatives are defined by
$$\triangle ^-_{\epsilon}f(x)=-\frac{f(x-\epsilon)-f(x)}{\epsilon} \quad \mbox{ and } \quad
 \triangle ^+_{\epsilon}f(x)=\frac{f(x+\epsilon)-f(x)}{\epsilon},$$
respectively. In short, we write $\triangle
^{\sigma}_{\epsilon}f(x)$, $\sigma=\pm$.
\end{Definition}

Next concept generalizes the derivative for continuous functions, not necessarily smooth.

\begin{Definition} (\textbf{cf.} \cite{Cresson2}) Let $f\in C^0$ and $\epsilon>0$. The $\epsilon$ scale derivative of $f$ at $x$ is defined by
\begin{equation}\label{scaleDerivative}\frac{\Box_{\epsilon}f}{\Box x}(x)=\frac12 (\triangle ^+_{\epsilon}f(x)+\triangle ^-_{\epsilon}f(x))-
i\frac12(\triangle ^+_{\epsilon}f(x)-\triangle ^-_{\epsilon}f(x)),
\quad i^2=-1.
\end{equation}
\end{Definition}

If $f$ is a $C^1$ function, and if we take the limit as
$\epsilon\to0$ in (\ref{scaleDerivative}), we obtain $f'(x)$. To simplify, when there is no
danger of confusion, we will write $\Box_{\epsilon}f$ instead of
$\Box_{\epsilon}f/\Box x$. For complex valued functions, we define
$$\frac{\Box_{\epsilon}f}{\Box x}(x)=\frac{\Box_{\epsilon}\mbox{Re}(f)}{\Box x}(x)+
i\frac{\Box_{\epsilon}\mbox{Im}(f)}{\Box x}(x).$$

We now collect the results needed to this work.
First the Leibniz rule for quantum calculus:

\begin{Theorem} (\textbf{cf.} \cite{Cresson2}) Given $f,g\in C^0$ and $\epsilon>0$, one has
\begin{equation}\label{LeibnizRule}\Box_{\epsilon}(f \cdot g)=\Box_{\epsilon}f \cdot g + f \cdot
\Box_{\epsilon}g+i\frac{\epsilon}{2}(\Box_{\epsilon} f \Box_{\epsilon}g-\boxminus_{\epsilon}f
\Box_{\epsilon}g-\Box_{\epsilon}f\boxminus_{\epsilon}g
-\boxminus_{\epsilon}f\boxminus_{\epsilon}g),
\end{equation}
where $\boxminus_{\epsilon}f$ is the complex conjugate of $\Box_{\epsilon}f$.
\end{Theorem}

If $f$ and $g$ are both differentiable, we obtain the Leibniz rule $(f\cdot g)'=f' \cdot g+ f \cdot g'$
from (\ref{LeibnizRule}), taking the limit as $\epsilon\to0$.

\begin{Definition}
Let $f\in C^0$, and $\alpha \in (0,1)$ be a real number.
We say that $f$ is H\"{o}lderian of H\"{o}lder exponent $\alpha$
if there exists a constant $c$ such that, for all $\epsilon>0$,
and all $x,x'\in \mathbb R$ such that $|x-x'|\leq \epsilon$,
$$|f(x)-f(x')|\leq c \epsilon^{\alpha}.$$
We denote by $H^{\alpha}$ the set of H\"{o}lderian functions
with H\"{o}lder exponent $\alpha$.
\end{Definition}

From now on, we assume that $\alpha\in(0,1)$ is fixed, and $\epsilon$ is a sufficiently small parameter,
$0<\epsilon \ll 1$. Let
$$C^{\alpha}_{\epsilon}(a,b) = \{ y:[a-\epsilon,b+\epsilon] \to \mathbb R \, | \, y \in H^{\alpha} \}.$$
A functional is a function $\Phi:C^{\alpha}_{\epsilon}(a,b)\to \mathbb C.$ We study the class of functionals $\Phi$ of the form
\begin{equation}
\label{Lagrangian}
\Phi(y)=\int_a^b f(x,y(x), \Box_{\epsilon}y(x))\, dx \, ,
\end{equation}
where $f:\mathbb R\times\mathbb R\times\mathbb C \to \mathbb C$ is a $C^1$ function, called the Lagrangian.
We assume that the Lagrangian satisfies
$$\left\|  Df(x,y(x), \Box_{\epsilon}y(x)) \right\| \leq C,$$
where $C$ is a positive constant, $D$ denotes the differential, and $\| \cdot \|$ is a norm for matrices.

If we consider the class of differentiable functions $y\in C^1$, we obtain the classical functional
$$\Phi(y)=\int_a^b f(x,y(x), \dot{y}(x))\, dx$$
of the calculus of variations when $\epsilon$ goes to zero.

The methods to solve problems of the calculus of variations admit a common variational approach: we consider a
class of functions $\eta(x)$
such that $\eta(a)=0=\eta(b)$; and admissible functions $\overline{y}=y+\epsilon_1 \eta$ on the neighborhood of $y$.
For $\epsilon_1$ sufficiently small, $\overline{y}$ is infinitely near $y$ and satisfies given boundary
conditions $\overline{y}(a)=y(a)$ and $\overline{y}(b)=y(b)$. For our purposes, we need another
assumption about functions $\eta$.

\begin{Definition} (\cite{Cresson2}) Let $y \in C^{\alpha}_{\epsilon}(a,b)$. A variation $\overline y$ of $y$ is a curve
of the form
$\overline{y}=y+h$, where $h\in C^{\beta}_{\epsilon}(a,b)$, $\beta \geq
\alpha 1 _{[1/2,1]}+(1-\alpha) 1 _{]0,1/2[}$, and $h(a)=0=h(b)$.
\end{Definition}

The minimal condition on $\beta$ is to ensure that the variation curve $\overline y$ is still on
$C^{\alpha}_{\epsilon}(a,b)$.

\begin{Definition} (\cite{Cresson2}) A functional $\Phi$ is called differentiable on $C^{\alpha}_{\epsilon}(a,b)$ if
for all variations $\overline y=y+h$, $h \in C^{\beta}_{\epsilon}(a,b)$,
$$\Phi(y+h)-\Phi(y)=F_y(h)+ R_y(h),$$
where $F_y$ is a linear operator and $R_y(h)=O(h^2)$.
\end{Definition}

\begin{Theorem} (\textbf{cf.} \cite{Cresson2}) For all $\epsilon>0$, the functional $\Phi$ defined by (\ref{Lagrangian}) is differentiable,
and its derivative is
$$F_y(h)=\int_a^b\left[ \frac{\partial f}{\partial y}(x,y(x), \Box_{\epsilon}y(x))-
\frac{\Box_{\epsilon}}{\Box x}\left( \frac{\partial f}{\partial \Box_{\epsilon}y} (x,y(x),
\Box_{\epsilon}y(x))  \right)  \right] h(x) \, dx$$
$$+\int_a^b \frac{\Box_{\epsilon}}{\Box x}\left( \frac{\partial f}{\partial \Box_{\epsilon}y} h(x)
\right)dx + i R_y(h)$$
with
$$R_y(h)=-\frac{\epsilon}{2} \int_a^b [ \Box_{\epsilon}f_{\epsilon}(x)\Box_{\epsilon}h(x) -
\boxminus_{\epsilon}f_{\epsilon}(x)\Box_{\epsilon}h(x)  -
\Box_{\epsilon}f_{\epsilon}(x)\boxminus_{\epsilon}h(x)$$
$$-\boxminus_{\epsilon}f_{\epsilon}(x)\boxminus_{\epsilon}h(x) ] \, dx$$
where
$$f_{\epsilon}(x)=\frac{\partial f}{\partial \Box_{\epsilon}y} (x,y(x), \Box_{\epsilon}y(x)).$$
\end{Theorem}

\begin{Definition}  (\cite{Cresson2}) Let $a_p(\epsilon)$ be a real or complex valued function, with parameter $p$.
We denote by $[\cdot]_{\epsilon}$ the (unique) linear operator defined by
$$a_p(\epsilon)-[a_p(\epsilon)]_{\epsilon} \to_{\epsilon \to 0}0 \quad \mbox{ and } \quad
[a_p(\epsilon)]_{\epsilon}=0 \, \mbox{ if } \lim_{\epsilon\to0}a_p(\epsilon)=0.$$
\end{Definition}

\begin{Definition} (\cite{Cresson2}) We say that $y$ is an extremal curve for the functional (\ref{Lagrangian})
on $C^{\beta}_{\epsilon}(a,b)$, 
if $[F_y(h)]_{\epsilon}=0$ for all $\epsilon>0$ and
$h \in C^{\beta}_{\epsilon}(a,b)$.
\end{Definition}

The main result of \cite{Cresson2} is a version of the Euler-Lagrange equation for nonsmooth curves:

\begin{Theorem} (\cite{Cresson2}) The curve $y$ is an extremal for the functional (\ref{Lagrangian}) on
$C^{\beta}_{\epsilon}(a,b)$ if and only if
$$\left[  \frac{\partial f}{\partial y}(x,y(x), \Box_{\epsilon}y(x))-
\frac{\Box_{\epsilon}}{\Box x}\left( \frac{\partial f}{\partial \Box_{\epsilon}y}
(x,y(x), \Box_{\epsilon}y(x))  \right) \right]_\epsilon=0$$
for every $\epsilon>0$.
\end{Theorem}


\section{Main results}
\label{sec:mr}

The isoperimetric problem is one 
of the most ancient optimization problems.
One seeks to find a continuously  differentiable curve $y=y(x)$, satisfying given boundary condition $y(a)=a_0$ and $y(b)=b_0$, 
which minimizes or maximizes a given functional
$$I(y)=\int_a^b f(x,y(x), \dot{y}(x))\, dx,$$
for which a second given functional
$$G(y)=\int_a^b g(x,y(x), \dot{y}(x))\, dx$$
possesses a given prescribed value $K$.
The classical method to solve this problem involves a Lagrange multiplier $\lambda$ and
consider the problem of extremizing the functional
$$\int_a^b(f-\lambda g) \, dx$$
using the respective Euler-Lagrange equation. In scale calculus we have an additional problem, because functionals $I$ and $G$ take complex values and so the Lagrange multiplier method must be
adapted. We will assume that $\| Dg(\cdot) \|$ is finite.

For our main theorem, we need the following lemma.

\begin{Lemma} If $\lim_{\epsilon\to0}(a_p(\epsilon))$ and
$\lim_{\epsilon\to0}(b_p(\epsilon))$ are both finite, then
$$[a_p(\epsilon)\cdot b_p(\epsilon)]_{\epsilon}=[a_p(\epsilon)]_{\epsilon}\cdot [b_p(\epsilon)]_{\epsilon}.$$
\end{Lemma}

\begin{proof} Since $\lim_{\epsilon\to0}(a_p(\epsilon))$ and
$\lim_{\epsilon\to0}(b_p(\epsilon))$ are finite, then
$\lim_{\epsilon\to0}[a_p(\epsilon)]_{\epsilon}$ and
$\lim_{\epsilon\to0}[b_p(\epsilon)]_{\epsilon}$ are also finite.
Moreover,
\begin{enumerate}
\item{$$\lim_{\epsilon\to0}(a_p(\epsilon)\cdot b_p(\epsilon)-
[a_p(\epsilon)]_{\epsilon}\cdot [b_p(\epsilon)]_{\epsilon})$$
$$=\lim_{\epsilon\to0}((a_p(\epsilon)-[a_p(\epsilon)]_{\epsilon})\cdot
b_p(\epsilon)+ [a_p(\epsilon)]_{\epsilon}\cdot(b_p(\epsilon)-
[b_p(\epsilon)]_{\epsilon})) =0.$$}
\item{If $\lim_{\epsilon\to0}(a_p(\epsilon)\cdot b_p(\epsilon))=0$, then
$\lim_{\epsilon\to0}(a_p(\epsilon))=0$ or
$\lim_{\epsilon\to0}(b_p(\epsilon))=0$. Therefore,
$[a_p(\epsilon)]_{\epsilon}=0$ or $[b_p(\epsilon)]_{\epsilon}=0$
and so $[a_p(\epsilon)]_{\epsilon}\cdot
[b_p(\epsilon)]_{\epsilon}=0$.}
\end{enumerate}
\end{proof}

\begin{Definition} Given a constraint functional $G(y)=K$ and a curve $\overline y$, we say that $\overline y$ is an extremal curve for the functional $I(y)=\int_a^b f(x,y(x), \Box_{\epsilon}y(x))\, dx$ subject to the constraint $G(y)=K$, if whenever $\hat y= \overline y + \sum_k h_k$,  $h_k \in C^{\beta}_{\epsilon}(a,b)$, is a variation satisfying the constraint $G(\hat y)=K$, then  
\begin{equation*}
\begin{split}
[F_{\overline y}(h_k)]_{\epsilon}
&= \int_a^b\left[ \frac{\partial f}{\partial y}(x,\overline{y}(x), \Box_{\epsilon}\overline{y}(x))-
\frac{\Box_{\epsilon}}{\Box x}\left( \frac{\partial f}{\partial \Box_{\epsilon}y} (x,\overline{y}(x),
\Box_{\epsilon}\overline{y}(x))  \right)  \right]_\epsilon h_k(x) \, dx \\
&=0
\end{split}
\end{equation*}
for all $\epsilon>0$ and for all $k$.
\end{Definition}

\begin{Theorem}\label{MainTeo} Let $\overline{y}\in C^{\alpha}_{\epsilon}(a,b)$. Suppose that
$\overline y$ is an extremal for the functional
$$\begin{array}{cccl}
I:  &  C^{\alpha}_{\epsilon}(a,b) & \to & \mathbb C \\
  & y  & \mapsto & \int_a^b f(x,y(x), \Box_{\epsilon}y(x))\, dx\\
\end{array}$$
on $C^{\beta}_{\epsilon}(a,b)$, subject to the boundary conditions
$y(a)=a_0$, $y(b)=b_0$ and the
integral constraint
$$G(y)= \int_a^b g(x,y(x), \Box_{\epsilon}y(x))\, dx=K \, ,$$
where $K \in \mathbb C$ is a given constant. If
\begin{enumerate}
\item{$\overline y$ is not an extremal for $G$;}
\item{and $\displaystyle \lim_{\epsilon\to0} \max_{x \in [a,b]} \left|
\left(  \frac{\partial f}{\partial y}- \frac{\Box_{\epsilon}}{\Box
x} \left(\left. \frac{\partial f}{\partial \Box_{\epsilon}y}
\right) \right) \right| _{(x,\overline y(x),
\Box_{\epsilon}\overline y(x))}  \right|$ and

$\displaystyle \lim_{\epsilon\to0} \max_{x \in [a,b]} \left|
\left( \frac{\partial g}{\partial y}- \frac{\Box_{\epsilon}}{\Box
x} \left(\left. \frac{\partial g}{\partial \Box_{\epsilon}y}
\right) \right) \right| _{(x,\overline y(x),
\Box_{\epsilon}\overline y(x))} \right|$ are both finite;}
\end{enumerate}
then there exists $\lambda \in \mathbb R$ such that
$$\left[ \left(  \frac{\partial L}{\partial y}- \frac{\Box_{\epsilon}}{\Box x}
\left(\left. \frac{\partial L}{\partial \Box_{\epsilon}y} \right)
\right) \right| _{(x,\overline y(x), \Box_{\epsilon}\overline
y(x))} \right]_{\epsilon}=0,$$ where $L=f-\lambda g$. In other
words, $\overline y$ is an extremal for $L$.
\end{Theorem}

\begin{Remark} Hypothesis~2 of Theorem~\ref{MainTeo} is trivially satisfied in the case where the admissible curves are smooth.
\end{Remark}

\begin{proof} To short, let $u=(x,\overline y(x), \Box_{\epsilon}\overline y(x))$.
Consider the two-parameter family of variations
$$\hat y = \overline y +\epsilon_1 \eta_1 + \epsilon_2 \eta_2,$$
such that $\eta_1, \eta_2 \in C^{\beta}_{\epsilon}(a,b)$, $\beta \geq
\alpha 1 _{[1/2,1]}+(1-\alpha) 1 _{]0,1/2[}$,
$\eta_1(a)=0=\eta_1(b)$, $\eta_2(a)=0=\eta_2(b)$, and $\epsilon_1,
\epsilon_2 \in B_r(0)$, with $r$ sufficiently small. Then, $\hat
y(a)=a_0$ and $\hat y(b)=b_0$, as prescribed, for all values of
the parameters $\epsilon_1$ and $\epsilon_2$. It is easy to see
that $\hat y \in C^{\alpha}_{\epsilon}(a,b)$.

\textbf{1.} If we fix two curves $\eta_1$ and $\eta_2$, we can
consider the functions $\overline I$ and $\overline G$ with two
variables $\epsilon_1$ and $\epsilon_2$, defined by
$$\overline I (\epsilon_1,\epsilon_2)=\int_a^b f(x,\overline{y}(x)+\epsilon_1\eta_1+\epsilon_2\eta_2,
\Box_{\epsilon}\overline{y}(x) +\epsilon_1\Box_{\epsilon}\eta_1+\epsilon_2\Box_{\epsilon}\eta_2 )\, dx$$
and
$$\overline G (\epsilon_1,\epsilon_2)=\int_a^b g(x,\overline{y}(x)+\epsilon_1\eta_1+\epsilon_2\eta_2,
\Box_{\epsilon}\overline{y}(x) +\epsilon_1\Box_{\epsilon}\eta_1+\epsilon_2\Box_{\epsilon}\eta_2 )\, dx.$$
Let $\overline{\overline G}=\overline G-K$.

\textbf{2.} We have $\nabla \overline{\overline G}(0,0)\not= 0$. Indeed, since g is a smooth function, $\overline{\overline G}$ is also smooth and
$$\begin{array}{ll}
\displaystyle\left.\frac{\partial \overline{\overline G}}{\partial \epsilon_1}\right|_{(0,0)}&
\displaystyle=\int_a^b \left( \eta_1 \left.\frac{\partial g}{\partial y}\right|_u+ \Box_{\epsilon}\eta_1
\left.\frac{\partial g}{\partial \Box_{\epsilon} y}\right|_u \right)\, dx\\
     & \displaystyle=\int_a^b \left( \left.\frac{\partial g}{\partial y}\right|_u -
     \frac{\Box_{\epsilon}}{\Box x} \left( \left.\frac{\partial g}{\partial
     \Box_{\epsilon} y}\right|_u\right) \right) \, \eta_1 \, dx\\
     & \quad \displaystyle+\int_a^b \frac{\Box_{\epsilon}}{\Box x} \left( \left.\frac{\partial g}{\partial
     \Box_{\epsilon} y}\right|_u \cdot \eta_1 \right) \, dx\\
     & \quad \displaystyle-i\frac{\epsilon}{2} \int_a^b[ \Box_{\epsilon} g_{\epsilon} \Box_{\epsilon}\eta_1-
     \boxminus_{\epsilon} g_{\epsilon} \Box_{\epsilon}\eta_1-\Box_{\epsilon} g_{\epsilon}
     \boxminus_{\epsilon}\eta_1-\boxminus_{\epsilon} g_{\epsilon}\boxminus_{\epsilon}\eta_1] \, dx \, ,\\
\end{array}$$
where
$$ g_{\epsilon}= \left.\frac{\partial g}{\partial  \Box_{\epsilon} y}\right|_u .$$
Since
$$\lim_{\epsilon\to 0} \int_a^b \frac{\Box_{\epsilon}}{\Box x} \left( \left.\frac{\partial g}{\partial
\Box_{\epsilon} y}\right|_u \cdot \eta_1(x) \right) \, dx=0$$
and
$$\lim_{\epsilon\to 0}\epsilon \int_a^b (Op_{\epsilon}g_{\epsilon} Op'_{\epsilon} \eta_1) \, dx=0,$$
where $Op_{\epsilon}$ and $Op'_{\epsilon}$ is equal to $\Box_{\epsilon}$ and $\boxminus_{\epsilon}$ (\textbf{cf.} \cite[Lemma~3.2]{Cresson2}),
it follows that
$$\left[ \left. \frac{\partial \overline{\overline G}}{\partial \epsilon_1} \right|_{(0,0)}  \right]_{\epsilon}=
\int_a^b\left[  \left.\frac{\partial g}{\partial y}\right|_u- \frac{\Box_{\epsilon}}{\Box x}\left(\left.
\frac{\partial g}{\partial \Box_{\epsilon}y}\right|_u \right) \right]_\epsilon \, \eta_1(x) \, dx.$$
Since $\overline y$ is not an extremal of $G$, there exists a curve $\eta_1$ such that
$$\left[ \left. \frac{\partial \overline{\overline G}}{\partial \epsilon_1} \right|_{(0,0)}
 \right]_{\epsilon} \not=0.$$
Therefore, by the definition of $[\, \cdot \,]_{\epsilon}$, we
conclude that
$$\left. \frac{\partial \overline{\overline G}}{\partial \epsilon_1} \right|_{(0,0)}\not=0.$$

\textbf{3.} We can choose $\epsilon_2\eta_2$ in order to satisfy
the isoperimetric condition. Since  $\nabla \overline{\overline
G}(0,0)\not= 0$ and  $\overline{\overline G}(0,0)= 0$, by the
implicit function theorem, there exists a function
$\epsilon_1:=\epsilon_1(\epsilon_2)$ defined on a neighbourhood of
zero such that
$$\overline{\overline G}(\epsilon_1(\epsilon_2),\epsilon_2)=0.$$

\textbf{4.} We now adapt the Lagrange multiplier method. Since $\overline{\overline
G}(\epsilon_1(\epsilon_2),\epsilon_2)=0$, for any $\epsilon_2$,
then
$$0=\frac{d}{d\epsilon_2}\overline{\overline G} (\epsilon_1(\epsilon_2),\epsilon_2)=
\frac{d \epsilon_1}{d \epsilon_2} \cdot \frac{\partial
\overline{\overline G}}{\partial \epsilon_1}+\frac{\partial
\overline{\overline G}}{\partial \epsilon_2}$$ and so, as
$\epsilon$ goes to zero,
$$\left.\frac{d \epsilon_1}{d \epsilon_2}\right|_0=-
\frac{\int_a^b\left( \left.\frac{\partial g}{\partial y}\right|_u-
\frac{\Box_{\epsilon}}{\Box x}\left(\left. \frac{\partial
g}{\partial \Box_{\epsilon}y}\right|_u \right) \right)\, \eta_2(x)
\, dx+\int_a^b \frac{\Box_{\epsilon}}{\Box x} \left(
\left.\frac{\partial g}{\partial \Box_{\epsilon} y}\right|_u \cdot
\eta_2 \right) \, dx-[ \ldots ]}
        {\int_a^b\left( \left.\frac{\partial g}{\partial y}\right|_u- \frac{\Box_{\epsilon}}{\Box x}\left(\left.
\frac{\partial g}{\partial \Box_{\epsilon}y}\right|_u \right)
\right) \, \eta_1(x) \, dx+\int_a^b \frac{\Box_{\epsilon}}{\Box x}
\left( \left.\frac{\partial g}{\partial \Box_{\epsilon}
y}\right|_u \cdot \eta_1 \right) \, dx-[\ldots]}
$$ is finite.
Observe that
$$\lim_{\epsilon\to0}\left.\frac{\partial \overline I}{\partial \epsilon_1}\right|_u=\lim_{\epsilon\to0}
\int_a^b\left(  \left.\frac{\partial f}{\partial y}\right|_u-
\frac{\Box_{\epsilon}}{\Box x}\left(\left. \frac{\partial
f}{\partial \Box_{\epsilon}y}\right|_u \right) \right) \,
\eta_1(x) \, dx$$ and
$$\lim_{\epsilon\to0}\left.\frac{\partial \overline{\overline G}}{\partial \epsilon_1}\right|_u=\lim_{\epsilon\to0}
\int_a^b\left(  \left.\frac{\partial g}{\partial y}\right|_u-
\frac{\Box_{\epsilon}}{\Box x}\left(\left. \frac{\partial
g}{\partial \Box_{\epsilon}y}\right|_u \right) \right)\, \eta_1(x)
\, dx$$ are also finite.
Let us prove that
\begin{equation}\label{LagrangeMulti.}\frac{d}{d\epsilon_2}
\left[\left.\overline{I}
(\epsilon_1(\epsilon_2),\epsilon_2)\right|_0\right]_{\epsilon}=0.
\end{equation}
A direct calculation shows that
$$\begin{array}{ll}
\displaystyle\frac{d}{d\epsilon_2}\left[\left.\overline{I}
(\epsilon_1(\epsilon_2),\epsilon_2)\right|_0 \right]_{\epsilon}& =
\displaystyle\left[\frac{d\epsilon_1}{d\epsilon_2} \frac{\partial
\overline I}
{\partial \epsilon_1}+\frac{\partial \overline I}{\partial \epsilon_2}\right]_{\epsilon}\\
&\displaystyle=\left[\frac{d\epsilon_1}{d\epsilon_2}\right]_{\epsilon}
\left[ \frac{\partial \overline I}
{\partial \epsilon_1}\right]_{\epsilon}+\left[ \frac{\partial \overline I}{\partial \epsilon_2}\right]_{\epsilon}\\
        &= \displaystyle\left[\frac{d\epsilon_1}{d\epsilon_2}\right]_{\epsilon} \int_a^b\left[ \frac{\partial f}{\partial y}-
        \frac{\Box_{\epsilon}}{\Box x}\left( \frac{\partial f}{\partial \Box_{\epsilon} y}  \right)
        \right]_{\epsilon}\, \eta_1 \, dx\\
        &\quad \displaystyle +\int_a^b \left[ \frac{\partial f}{\partial y}- \frac{\Box_{\epsilon}}{\Box x}
        \left( \frac{\partial f}{\partial \Box_{\epsilon} y}  \right)  \right]_{\epsilon}\, \eta_2 \, dx\\
        & =  \displaystyle\left[\frac{d\epsilon_1}{d\epsilon_2}\right]_{\epsilon} \, \left[F_{\overline y}(\eta_1)\right]_{\epsilon}+\left[F_{\overline y}(\eta_2)\right]_{\epsilon}=0\\
        \end{array}$$
since $\overline y$ is an extremal of $I$ subject to the constraint $G=K$. On the other hand, for any $\epsilon_2$, we also have
$\left[ \overline{\overline
G}(\epsilon_1(\epsilon_2),\epsilon_2)\right]_{\epsilon}=0$.
 Therefore,
$$0=\frac{d}{d\epsilon_2}\left[\overline{\overline G} (\epsilon_1(\epsilon_2),\epsilon_2)\right]_{\epsilon}=
\left[\frac{d \epsilon_1}{d \epsilon_2}\right]_{\epsilon} \cdot
\left[ \frac{\partial \overline{\overline G}}{\partial \epsilon_1}
\right]_{\epsilon}+ \left[ \frac{\partial \overline{\overline
G}}{\partial \epsilon_2}\right]_{\epsilon}$$ and so
$$\left[\frac{d \epsilon_1}{d \epsilon_2}\right]_{\epsilon}=-\frac{\left[ \frac{\partial \overline{\overline G}}{\partial \epsilon_2}
\right]_{\epsilon}}{\left[ \frac{\partial \overline{\overline G}}{\partial \epsilon_1}\right]_{\epsilon}}.$$
Using condition (\ref{LagrangeMulti.}),
we have
$$\left|
\begin{array}{cc}
\left[ \frac{\partial \overline{I}}{\partial \epsilon_1}
\right]_{\epsilon}& \left[ \frac{\partial
\overline{\overline G}}{\partial \epsilon_1}\right]_{\epsilon}\\
 \left[ \frac{\partial \overline{I}}{\partial \epsilon_2}\right]_{\epsilon}&\left[ \frac{\partial
 \overline{\overline G}}{\partial \epsilon_2}\right]_{\epsilon}
\end{array}\right|=0.$$
Since $\left[ \frac{\partial \overline{\overline G}}{\partial
\epsilon_1}\right]_{\epsilon}\not=0$, we conclude that there
exists some real $\lambda$ such that
$$\left( \left[ \frac{\partial \overline{I}}{\partial \epsilon_1} \right]_{\epsilon},
\left[ \frac{\partial \overline{I}}{\partial \epsilon_2}
\right]_{\epsilon}\right) = \lambda \left( \left[ \frac{\partial
\overline{\overline G}}{\partial \epsilon_1}\right]_{\epsilon},
\left[ \frac{\partial \overline{\overline G}}{\partial
\epsilon_2}\right]_{\epsilon} \right).$$

\textbf{5.} In conclusion, since
\begin{equation*}
\begin{split}
0 & =\left[\left.\frac{\partial}{\partial \epsilon_2}(\overline I - \lambda \overline{\overline G})\right|_{(0,0)}\right]_{\epsilon}\\
&=\int_a^b \left[\eta_2 \left.\frac{\partial f}{\partial
y}\right|_u+\Box_{\epsilon}\eta_2 \left. \frac{\partial
f}{\partial \Box_{\epsilon} y}\right|_u - \lambda \left( \eta_2
\left.\frac{\partial g}{\partial y}\right|_u+\Box_{\epsilon}\eta_2
\left. \frac{\partial g}{\partial \Box_{\epsilon} y}\right|_u
\right) \right]_{\epsilon}  \, dx\\
&=\int_a^b \left[ \left.\frac{\partial f}{\partial y}\right|_u- \frac{\Box_{\epsilon}}{\Box x}
\left(\left. \frac{\partial f}{\partial \Box_{\epsilon}
y}\right|_u\right)  -\lambda \left( \left.\frac{\partial
g}{\partial y}\right|_u- \frac{\Box_{\epsilon}}{\Box x} \left(
\left.\frac{\partial g}{\partial \Box_{\epsilon} y}\right|_u
\right)  \right)\right]_{\epsilon}\eta_2 \, dx\\
&=\int_a^b \left[ \left.\frac{\partial L}{\partial y}\right|_u-\frac{\Box_{\epsilon}}{\Box x}
\left( \left.\frac{\partial L}{\partial \Box_{\epsilon}
y}\right|_u \right) \right]_{\epsilon}\eta_2 \, dx
\end{split}
\end{equation*}
and $\eta_2$ is any curve, we obtain
$$\left[ \left.\frac{\partial L}{\partial y}\right|_u- \frac{\Box_{\epsilon}}{\Box x}
\left( \left.\frac{\partial L}{\partial \Box_{\epsilon} y}\right|_u \right) \right]_{\epsilon}=0.$$
\end{proof}


\section{An example}
\label{sec:ex}

Let $f(x,y,v)=(v-\frac{\Box_{\epsilon}}{\Box x} |x|)^2$. With simple calculations, one proves that
$$\frac{\Box_{\epsilon}}{\Box x} |x|= \left\{
\begin{array}{ll}
1 & \mbox{ if } x \geq \epsilon\\
x/\epsilon-i(\epsilon-x)/\epsilon & \mbox{ if } 0 \leq x <\epsilon\\
x/\epsilon-i(\epsilon+x)/\epsilon & \mbox{ if } -\epsilon< x<0\\
-1 & \mbox{ if } x \leq -\epsilon\\
\end{array}
  \right.$$
Suppose we want to find the extremals for the functional
\begin{equation}
\label{example}
\int_{-1}^1f(x,y(x), \Box_{\epsilon}y(x))\, dx
\end{equation}
subject to the integral constraint
$$\int_{-1}^1 g(x,y(x), \Box_{\epsilon}y(x))\, dx=\frac23,$$
where $g(x,y,v)=x+y^2$, and to the boundary conditions
$y(-1)=1=y(1)$. The (nonsmooth) curve $y=|x|$ satisfies the constraint integral, and the following conditions:
\begin{enumerate}
\item{$\left[ \frac{\partial f}{\partial y}- \frac{\Box_{\epsilon}}{\Box x}
\left(\left. \frac{\partial f}{\partial \Box_{\epsilon}y} \right)
\right)  \right]_{\epsilon}=0$:

$\frac{\partial f}{\partial y}- \frac{\Box_{\epsilon}}{\Box x}
\left(\frac{\partial f}{\partial \Box_{\epsilon}y} \right)
=-\frac{\Box_{\epsilon}}{\Box x}
\left(2 \left(  \frac{\Box_{\epsilon}}{\Box x} |x|-\frac{\Box_{\epsilon}}{\Box x} |x|  \right)\right)=0$.}
\item{$\left[ \frac{\partial g}{\partial y}- \frac{\Box_{\epsilon}}{\Box x}
\left(\left. \frac{\partial g}{\partial \Box_{\epsilon}y} \right)
\right)  \right]_{\epsilon}\not=0$:

$\frac{\partial g}{\partial y}- \frac{\Box_{\epsilon}}{\Box x}
\left(\frac{\partial g}{\partial \Box_{\epsilon}y} \right)
=2|x|$.}
\item{$\lim_{\epsilon\to0} \max_{x \in [-1,1]} \left|
\frac{\partial f}{\partial y}- \frac{\Box_{\epsilon}}{\Box
x} \left(\left. \frac{\partial f}{\partial \Box_{\epsilon}y}
\right) \right) \right|= \lim_{\epsilon\to0} 0=0$.}
\item{$\lim_{\epsilon\to0} \max_{x \in [-1,1]} \left|
\frac{\partial g}{\partial y}- \frac{\Box_{\epsilon}}{\Box
x} \left(\left. \frac{\partial g}{\partial \Box_{\epsilon}y}
\right) \right) \right|=\lim_{\epsilon\to0} 2=2.$}
\end{enumerate}

Observe that, since $y=|x|$ is actually an extremal of (\ref{example}), we may take $\lambda=0$.


\section*{Acknowledgments}

Work supported by {\it Centre for
Research on Optimization and Control} (CEOC)
from the ``Funda\c{c}\~{a}o para a Ci\^{e}ncia e a Tecnologia''
(FCT), cofinanced by the European Community Fund FEDER/POCI
2010.



\end{document}